\documentclass[a4paper,reqno,11pt]{amsart}
\usepackage[a4paper,margin=1in]{geometry}
\usepackage{graphicx}
\usepackage{amsmath,amssymb,amsthm}
\usepackage{hyperref}

\theoremstyle{plain}
\newtheorem{theorem}{Theorem}[section]

\theoremstyle{definition}

\numberwithin{equation}{section}

\newcommand{\R}{\mathbb{R}}
\newcommand{\abs}[2][]{#1\lvert #2 #1\rvert}

\newcommand{\FF}{{\mathcal S}}

\newcommand{\w}{w^{(d)}}

\title{Improved bound in the Benjamin and Lighthill conjecture}

\author{Evgeniy Lokharu}
\address{Department of Mathematics, Link\"oping University, SE-581 83 Link\"oping, Sweden}

\begin{document}
	
	\begin{abstract}
		
	The classical Benjamin and Lighthill conjecture about steady water waves states that the non-dimensional flow force constant of a solution is bounded by the corresponding constants of the supercritical and subcritical uniform streams respectively. These inequalities determine a parameter region that covers all steady motions. In fact not all points of the region determine a steady wave. In this note we prove a new and explicit lower bound for the flow force constant, which is asymptotically sharp in a certain sense. In particular, this recovers the well known inequality $F<2$ for the Froude number, while significantly reducing the parameter region supporting steady waves. 
		
	\end{abstract}
	
	\maketitle

	\section{Introduction} \label{s:introduction}

	We consider the classical problem for steady waves on the surface of an ideal fluid above a flat bottom. In non-dimensional variables, when the mass flux and the gravitational constant are scaled to $1$, every flow has two constants of motion, which are the total head $r$ and the flow-force $\FF$. In 1954 Benjamin and Lighthill \cite{Benjamin1954a} made a conjecture about the independent possible ranges of $r$ and $\FF$. Specifically, every steady flow, irrespective of amplitude or wavelength, is to realize a point $(r,\FF)$ within a certain region in $(r,\FF)$-plane. The latter is determined by all points $(r,\FF)$ for which 
	\begin{equation} \label{BLconj}
	\FF_-(r) \leq \FF \leq \FF_+(r),
	\end{equation}
	where $\FF_+(r)$ and $\FF_-(r)$ are the flow force constants corresponding to the subcritical and supercritical laminar flows respectively. The boundary of this region is a cusped curve representing all uniform streams (see Figure \ref{fig:BL}). This conjecture was verified by Benjamin for all irrotational Stokes waves in \cite{Benjamin95}. The left inequality in \eqref{BLconj} was obtained earlier by Keady \& Norbury \cite{Keady1975} (also for periodic wavetrains). Recently Kozlov and Kuznetsov \cite{Kozlov2009a,Kozlov2011a} proved \eqref{BLconj} for arbitrary solutions under weak regularity assumptions, provided the Bernoulli constant $r$ is close to it's critical value $r_c$; it was extended to the rotational setting in \cite{Kozlov2017a}, again for $r \approx r_c$; the latter condition guarantees that solutions are of small amplitude. The left inequality in \eqref{BLconj} for periodic waves with a favorable vorticity was obtained by Keady \& Norbury \cite{Keady1978}. It was recently proved in \cite{LokActa} that \eqref{BLconj} remains true for arbitrary waves with vorticity. In particular, all irrotational steady waves (not necessarily periodic) correspond to points in the region.
	
	As it was recently proved in \cite{KozLokhWheeler2020} that all solitary waves correspond to points on a part of lower boundary $\FF = \FF_-(r)$ (dashed curve in Figure \ref{fig:BL}). It is believed (but unproved) that the dashed curve ends with the point that determines the highest solitary wave. Anyway, the well known bound $F<2$ (see \cite{starr}) for the Froude number implies that no solitary waves presented on $\FF = \FF_-(r)$ for large $r$. Besides, there are no small-amplitude periodic waves with $\FF \sim \FF_-(r)$. Together this suggests that no steady waves are determined by points in a neighborhood of the lower boundary $\FF = \FF_-(r)$ when $r$ is large.
	
	In the present paper we prove the existence of a barrier, which is explicitly given by $\FF = \tfrac12 r^2$. This barrier intersects the lower boundary at a point with $F=2$ and is asymptotically close to the upper boundary $\FF = \FF_+(r)$, which shows that all steady waves, including waves of greatest height, correspond to a smaller part of the region (grey region in Figure \ref{fig:BL}) from the Benjamin and Lighthill conjecture. Our proof is of special interest by itself. It is based on the flow force function formulation, introduced recently in \cite{Basu2020}. This new reformulation is very similar to the problem for steady waves with positive constant vorticity. Our idea is essentially based on that similarity. It is known from \cite{Kozlov2015} that the Bernoulli constant for waves with positive vorticity is bounded from above; see also \cite{Lokharu2020} for the case of negative vorticity. Thus, using a similar argument we can show that the flow force function formulation (after a rescaling) admit an explicit upper bound for the Bernoulli constant, which is equivalent to the desired inequality $\FF > \tfrac12 r^2$.

\begin{figure}[!t]
	\centering%
	\includegraphics[scale=0.5]{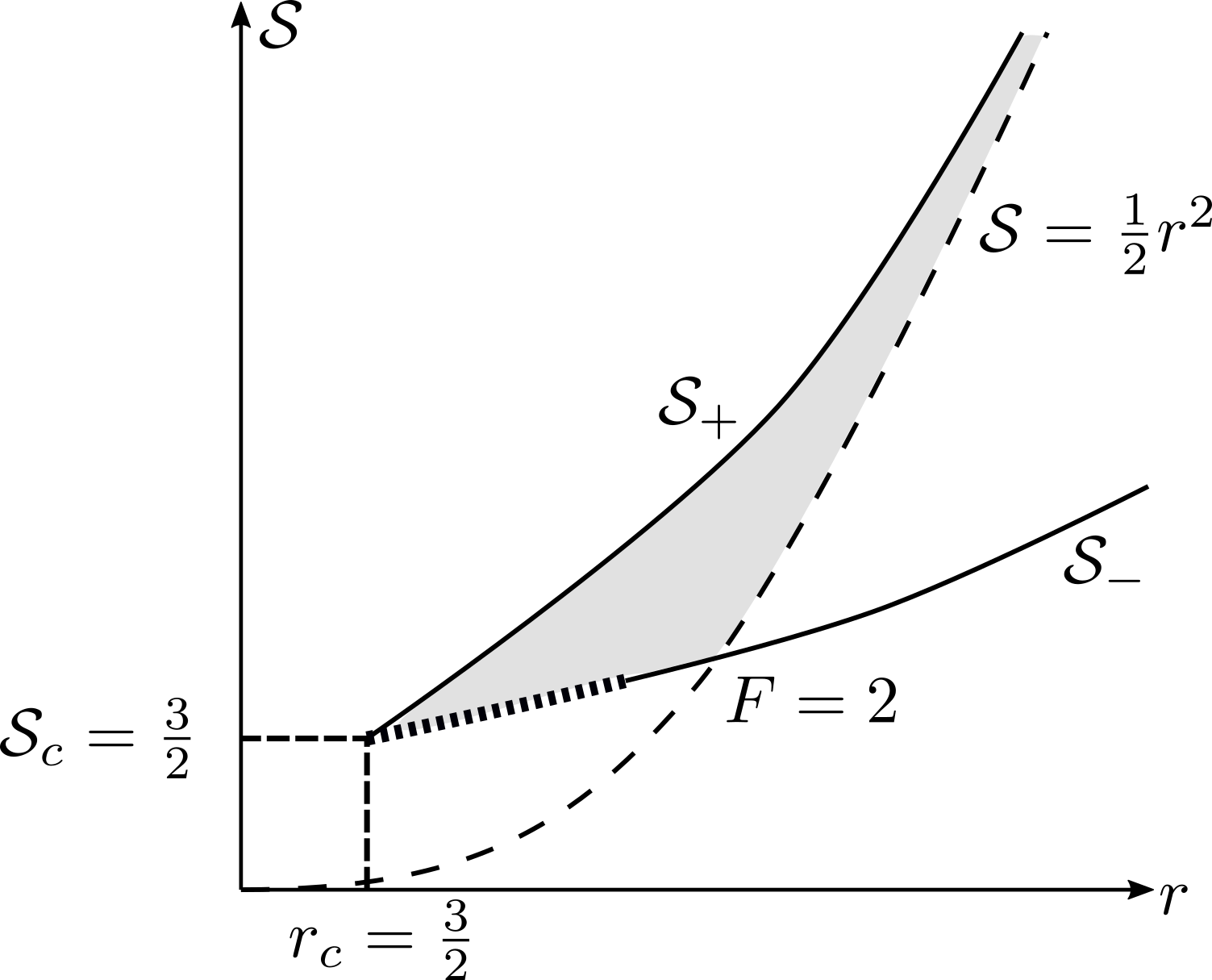}
	\caption{Parameter region for steady water waves}
	\label{fig:BL}
\end{figure}

	\section{Statement of the problem} \label{s:statement}
	
	A non-dimensional version of the stream function formulation for two-dimensional water waves has the following form (see \cite{Benjamin95}):
	
	\begin{subequations}\label{sys:stream}
	\begin{alignat}{2}
	\label{sys:stream:lap}
	\Delta\psi&=0 &\qquad& \text{for } 0 < y < \eta,\\
	\label{sys:stream:bern}
	\tfrac 12\abs{\nabla\psi}^2 +  y  &= r &\quad& \text{on }y=\eta,\\
	\label{sys:stream:kintop} 
	\psi  &= 1 &\quad& \text{on }y=\eta,\\
	\label{sys:stream:kinbot} 
	\psi  &= 0 &\quad& \text{on }y=0.
	\end{alignat}
	\end{subequations}	

	Here $r$ is referred to as the Bernoulli constant, problem's parameter. The latter problem admits another constant (see \cite{BENJAMIN1984} for more details), the flow force, defined as 
	\begin{equation} \label{flowforce}
		\FF = \int_0^{\eta}(\tfrac12(\psi_y^2 - \psi_x^2) - y + r )\, dy.
	\end{equation}	
	After taking $x$-derivative in \eqref{flowforce} and using \eqref{sys:stream:lap} together with the boundary relation \eqref{sys:stream:bern}, one verifies that $\FF$ is a constant of motion independent of $x$. Our main result can now be stated as follows.

	\begin{theorem}\label{thm:main}
		Let $(\psi,\eta) \in C^{2,\gamma}(\overline{D_\eta}) \times C^{2,\gamma}(\R)$ be a solution to \eqref{sys:stream} which is not a laminar flow. Then $\FF > \tfrac12 r^2$, where $r$ and $\FF$ are the Bernoulli and the flow force constants respectively, defined by \eqref{sys:stream:bern} and \eqref{flowforce}.
	\end{theorem}
	
	Let us compare the quantity $\FF_+(r)$ and the bound $\tfrac12 r^2$ from the theorem. Note that
	\[
	\FF_+(r) = \frac{1}{2 d_+(r)} - \tfrac12 d_+^2(r) + r d_+(r),
	\]
	where $d_+(r) > 1$ is a root of $\tfrac12 d^{-2} + d = r$. Thus, for large $r$, we find $d_+(r) \sim r$ and more precisely,
	\[
	d_+(r) = r - \frac{1}{2 r^2} + O(r^{-3}) \ \ \text{as} \ \ r \to +\infty.
	\]
	Therefore, we obtain
	\[
	\begin{split}
	\FF_+(r) & = \frac{1}{2r} \left(1 + O(r^{-3})\right) - \tfrac12 \left( r^2 - \frac{1}{r} + O(r^{-2}) \right) + r \left(r - \frac{1}{2r^2} + O(r^{-3})\right) \\
	& = \frac{r^2}{2} + \frac{1}{2r} + O(r^{-2}).
	\end{split}	
	\]
	We see that the curve $\FF = \tfrac12 r^2$ from Theorem \ref{thm:main} is below the upper boundary $\FF = \FF_+(r)$ (see Figure \ref{fig:BL}) for large $r$ and is asymptotically accurate. On the other hand, curves $\FF = \FF_-(r)$ and $\FF = \tfrac12 r^2$ have one point of intersection, for which the Froude number $F:=d^{-3/2} = 2$. This recovers the well known bound $F < 2$ for solitary waves; see \cite{starr, AmickToland81b, mcleod}.
	
	It is quite interesting to find that for large $r$ an arbitrary solution, including highest waves, must have the flow force constant close to $\FF_+(r)$. At this point we can make a hypothesis that all steady waves for large $r$ are of small-amplitude.
	
	\section{Proof of Theorem \ref{thm:main}} \label{s:proof}

	\subsection{Flow force function formulation} \label{s:proof:ff}	
	
	Based on the definition for the flow force constant \eqref{flowforce}, we introduce the corresponding flow force function
	\begin{equation} \label{fff}
		F(x,y) = \int_0^{y}(\tfrac12(\psi_y^2(x,y') - \psi_x^2(x,y')) - y' + r )\, dy'.
	\end{equation}
	Just as in \cite{Basu2020} we can reformulate the water wave problem in terms of the function $F$. It is straightforward to obtain
	\begin{equation} \label{fff:grad}
	F_x = \psi_x \psi_y, \ \ F_y = \tfrac12(\psi_y^2 - \psi_x^2) - y + r.
	\end{equation}
	Thus, we arrive to an equivalent formulation given by
	\begin{subequations}\label{sys:fff}
	\begin{alignat}{2}
	\label{sys:fff:lap}
	\Delta F + 1&=0 &\qquad& \text{for } 0 < y < \eta,\\
	\label{sys:fff:bern}
	\tfrac 12 (-\eta' F_x + F_y) +  y  &= r &\quad& \text{on }y=\eta,\\
	\label{sys:fff:kintop} 
	F  &= \FF &\quad& \text{on }y=\eta,\\
	\label{sys:fff:kinbot} 
	F  &= 0 &\quad& \text{on }y=0.
	\end{alignat}
	\end{subequations}
	It follows immediately from the maximum principle that $F > 0$ for $y>0$. In fact one can show that
	\begin{equation}\label{ff:uni}
	F_y \geq \psi_y^2 > 0 \ \ \text{for } 0 \leq y \leq \eta.
	\end{equation}
	To see that it is enough to apply the maximum principle to the superharmonic function $\Phi=\tfrac12(\psi_x^2 + \psi_y^2) + y$. It is straightforward that $\Phi_y = 0$ on $y$ and $\Phi = r$ on $y = \eta$, so that $\Phi \leq r$ everywhere. This yields \eqref{ff:uni}.
	
	In what follows we want to treat the system \eqref{sys:fff} as the water wave problem with constant vorticity. Thus, it is convenient to have the "mass flux" $F = \FF$ scaled to $1$. This suggests new variables
	\[
	X = \sqrt{\FF}^{-1} x, \ \ Y = \sqrt{\FF}^{-1} y, \ \ \zeta(X) = \sqrt{\FF}^{-1} \eta(x), \ \ \bar{F}(X,Y) = \FF^{-1} F(x,y).
	\]
	The scaled problems is
	\begin{subequations}\label{sys:ffbar}
	\begin{alignat}{2}
	\label{sys:ffbar:lap}
	\Delta \bar{F} + 1&=0 &\qquad& \text{for } 0 < Y < \zeta,\\
	\label{sys:ffbar:bern}
	\tfrac 12 (-\zeta_X \bar{F}_X + \bar{F}_Y) +  Y  &= R:=r \sqrt{\FF}^{-1} &\quad& \text{on }Y=\zeta,\\
	\label{sys:ffbar:kintop} 
	\bar{F}  &= 1 &\quad& \text{on }Y=\zeta,\\
	\label{sys:ffbar:kinbot} 
	\bar{F}  &= 0 &\quad& \text{on }Y=0.
	\end{alignat}
	Furthermore, in view of \eqref{ff:uni}, we additionally have
	\begin{equation} \label{uni:ffbar}
		F_Y > 0 \ \ \text{for} \ \ 0 \leq  Y \leq \zeta.
	\end{equation}	
	\end{subequations}	
	We are going to prove certain bounds for the "Bernoulli constant" $R$ in \eqref{sys:ffbar:bern}. Note that the system \eqref{sys:ffbar} is very similar to the stream function formulation of the water problem with constant vorticity, for which the desired bounds were obtained in \cite{Kozlov2015}. Thus, a similar argument can be applied here and we adapt it below.
		
	\subsection{Stream solutions} \label{s:stream}
	
	In order to obtain bounds for $R$ we need to study stream solutions to \eqref{sys:ffbar}. These are pairs $\bar{F} = U(Y;d), \ \zeta(X) = d$, parametrized by a constant "depth" $d > 0$. Using this ansats in \eqref{sys:ffbar}, one finds
	\[
	U(Y;d) = - \tfrac12 Y^2 + (d^{-1} + \tfrac12 d) Y.
	\]
	The corresponding Bernoulli constant is given by
	\[
	R(d) = \frac{1}{2d} + \frac{3d}{4}.
	\]
	We are interested in unidirectional solutions only, for which $U_Y > 0$ on $[0,d]$. As a result we obtain a restriction on $d$:
	\[
	0 < d < d_0 = \sqrt{2}.
	\]
	This critical value is characterized by the relation $U_Y(d_0;d_0) = 0$, while $U_Y(Y;d_0) > 0$ for $Y \in [0,d_0)$. Let us denote 
	\[
	R_0 = R(d_0) = \sqrt{2} \ \ \text{and} \ \ R_c = \sqrt{\frac32}, \ \ d_c = \sqrt{\frac23}.
	\]
	Note that $R_c = R(d_c)$ is the global minimum of $R(d)$. Thus, for any $R \in (R_c,R_0)$ there are two solutions $d = d_-(R)$ and $d = R_+(d)$ with $d_-(R) < d_+(R)$ to the equation
	\[
	R = R(d).
	\]	
	These depths are similar to the subcritical and supercritical depths of conjugate laminar flows of the original water wave problem.
	\subsection{Bounds for the Bernoulli constant} \label{s:bounds}
	Our aim is to prove the following theorem:
	\begin{theorem}\label{thm:bounds}
		Let $(\bar{F},\zeta)$ be an arbitrary non-trivial (other than a stream solution) solution to \eqref{sys:ffbar}. Then the corresponding Bernoulli constant $R$ is subject to the inequalities $R_c < R < R_0$.
	\end{theorem}	
	
	Note that the statement of Theorem \ref{thm:main} follows directly from the upper bound $R < R_0$. Indeed, a non-trivial solution $(\psi,\eta)$ of the original system \eqref{sys:stream} with the Bernoulli constant $r$ generates a solution to \eqref{sys:ffbar} with $R = r\sqrt{\FF}^{-1}$. Now Theorem \ref{thm:bounds} gives $r\sqrt{\FF}^{-1} < R_0 = \sqrt{2}$, which is equivalent to $\FF > \tfrac12 r^2$ as stated in Theorem \ref{thm:main}.

	\begin{proof}[Proof of Theorem \ref{thm:bounds}]
		Our argument is based on a comparison of a given solution $(\bar{F},\zeta)$ to different stream solutions from Section \ref{s:stream}. For this purpose we will apply the partial hodograph transformation introduced by Dubreil-Jacotin \cite{DubreilJacotin34} but for the flow force function formulation \eqref{sys:ffbar}. This is possible since $\bar{F}_Y > 0$ everywhere by \eqref{ff:uni}. Thus, we introduce new independent variables
		\[
		q = X, \ \ p = \bar{F}(X,Y),
		\]
		while new unknown function $h(q,p)$ (height function) is defined from the identity
		\[
		h(q,p) = y.
		\]
		Note that it is related to the flow force function $\bar{F}$ through the formulas
		\begin{equation} \label{height:stream}
		\bar{F}_X = - \frac{h_q}{h_p}, \ \ \bar{F}_Y = \frac{1}{h_p},
		\end{equation}
		where 
		\begin{subequations}\label{sys:h}
			\begin{equation} \label{sys:h:uni}
			h_p > 0
			\end{equation}
			throughout new fixed domain $S = \R \times (0,1)$. An equivalent problem for $h(q,p)$ is given by
			\begin{alignat}{2}
			\label{sys:h:main}
			\left( \frac{1+h_q^2}{2h_p^2} + p \right)_p - \left(\frac{h_q}{h_p}\right)_q &=0 &\qquad& \text{in } S,\\
			\label{sys:h:top}
			\frac{1+h_q^2}{2h_p} +  h  &= R &\quad& \text{on }p=1,\\
			\label{sys:h:bot} 
			h  &= 0 &\quad& \text{on }p=0.
			\end{alignat}
		\end{subequations}
		The surface profile $\zeta$ becomes the boundary value of $h$ on $p = 1$: 
		\[
		h(q,1) = \zeta(q), \ \ q \in \R.
		\]
		For a detailed derivation of \eqref{sys:h} we refer to \cite{Basu2020}. Applying a similar transformation for  stream functions $U(Y;d)$, we obtain the corresponding height functions $H(p;d)$, subject to
		\[
		\left(\frac{1}{2H_p^2} + p\right)_p = 0, \ \ H(0) = 0, \ \ H(1) = d, \ \ \frac{1}{2 H_p(1;d)} + d = R(d).
		\]		
		Now because the domain for $h$ and $H$ is the same, we can compare these functions using the maximum principle. More precisely, we put
		\[
		w^{(d)} = h - H
		\]
		and using the corresponding equations for $h$ and $H$ one finds that $\w$ solves a homogeneous elliptic equation
		\begin{equation}\label{eqn:w}
			\frac{1+h_q^2}{h_p^2} \w_{pp} - 2\frac{h_q}{h_p} \w_{qp} + \w_{qq} - \w_p + \frac{(\w_q)^2 H_{pp}}{h_p^2} - \frac{\w_p (h_p + H_p) H_{pp}}{h_p^2 H_p^2} = 0.
		\end{equation}
		Thus, every $\w$ is subject to a maximum principle; see \cite{Vitolo2007} for an elliptic maximum principle in unbounded domains.
		
		Now we can prove Theorem \ref{thm:bounds}. Note that when $d \to d_0 = \sqrt{2}$ we have $H_p(1;d) \to +\infty$. Thus, we can choose $d < \sqrt{2}$ for which $\w_p < 0$ everywhere on $p=1$. Therefore, $\w < 0$ by the maximum principle, because otherwise we would obtain a contradiction with the Hopf lemma. Thus, we have $\zeta(q) < d$ for all $q \in \R$. By the continuity we can find $d_\star < d <  d_0$ such that $\sup_{\R} \zeta = d_\star$. In this case we can always find a sequence $q_j$, possibly unbounded, such that
		\[
		\lim_{j \to +\infty} \zeta(q_j) = d_\star, \ \ 	\lim_{j \to +\infty} w_q^{(d_\star)}(q_j,1) = 0, \ \ \lim_{j \to +\infty} w_p^{(d_\star)}(q_j,1) \geq 0.
		\]
		Taking the corresponding limit in \eqref{sys:h:top} and using relations from above we obtain $R \leq R(d_\star) < R_0$. The bottom inequality $R>R_c$ can be proved just the same way by choosing $d_\dagger = \inf_{\R} \zeta$ and repeating a similar argument for $w^{(d_\dagger)}$ (note that $\zeta$ is separated from zero).
	\end{proof}
	\bibliographystyle{siam}
	\bibliography{bibliography}
\end{document}